\documentclass[a4paper,12pt,reqno]{amsart}

 \usepackage[T2A]{fontenc}
 \usepackage[utf8]{inputenc}
 \usepackage[ukrainian,english]{babel}
 \usepackage{amssymb,upref}
\usepackage{amsthm}

\usepackage{graphicx}
\usepackage{mathrsfs}

\usepackage[text={130mm,200mm}]{geometry}

\theoremstyle{plain}
\newtheorem{theorem}{Theorem}
\newtheorem*{theorem*}{Theorem}
\newtheorem{corollary}{Corollary}
\newtheorem*{corollary*}{Corollary}
\newtheorem{lemma}{Lemma}
\newtheorem*{lemma*}{Lemma}

\newtheorem*{proposition*}{Proposition}

\newtheorem*{conjecture*}{Conjecture}
\theoremstyle{definition}

\newtheorem*{definition*}{Definition}
\theoremstyle{remark}

\newtheorem*{remark*}{Remark}

\begin{document}

\title[On one class of nowhere non-monotonic functions]{On one class of nowhere non-monotonic functions with fractal properties that contains a subclass of singular functions}

\author{S. O. Klymchuk}
\address[S. O. Klymchuk]{Institute of Mathematics of NAS of Ukraine, Kyiv, Ukraine\\
ORCID 0009-0005-3979-4543}
\email{svetaklymchuk@imath.kiev.ua}

\author{M. V. Pratsiovytyi}
\address[M. V. Pratsiovytyi]{Institute of Mathematics of NAS of Ukraine,  Dragomanov Ukrainian State University, Kyiv, Ukraine\\
ORCID 0000-0001-6130-9413}
\email{prats4444@gmail.com}

\subjclass{11K55, 26A27, 26A30}

\keywords{continuous non monotonic function, singular function, level set of function, differentiability of function, $Q^*_3$-representation of real number, geometry of $Q^*_3$-representation, cylinders, sets of Cantor type.}

\thanks{Collection of Proceedings of the Institute of Mathematics of the NAS of Ukraine. -- Kyiv: Institute of Mathematics of the NAS of Ukraine, 2017, p. 23--37. }

\begin{abstract}
We study one class of continuous functions $f$ defined on segment $[0,1]$ by equality
$$
f(x)=\delta_{\alpha_1(x)1}+\sum^{\infty}_{k=2}\left[\delta_{\alpha_k(x)k}\prod^{k-1}_{j=1}g_{\alpha_j (x)j}\right]\equiv\Delta^{G^*_3}_{\alpha_1\alpha_2\ldots\alpha_k\ldots},
$$
where $||q^*_{ik}||$ is given infinite stochastic positive matrix ($i=0,1,2$; $k \in N$); $\beta_{0k}=0$, $\beta_{1k}=q_{0k}$, $\beta_{2k}=q_{0k}+q_{1k}$;
$(\varepsilon_k)$ is given sequence of numbers such that $0\leqslant \varepsilon_k \leqslant 1 $; $g_{0k}=\dfrac{1+\varepsilon_k}{3}=g_{2k}$, $g_ {1k}=\dfrac{1-2\varepsilon_k}{3}$, $\delta_{0k}=0$, $\delta_{1k}=g_{0k}$, $\delta_{2k}=g_{0k}+g_{1k}$, $k\in N$.

We found criteria of strict monotonicity, non monotonicity and nowhere monotonicity, non-differentiability and singularity of the functions. We pay attention to properties of level sets of the functions.
\end{abstract}

\maketitle
\section{Introduction}
Let $A_3=\{0,1,2\}$ denote the alphabet of the ternary ($3$--adic) numeral system, $L_3=A_3\times A_3\times ...$ denote the space of sequences over this alphabet and $Q_3^*=|q_{ik}|$ denote an infinite positive stochastic matrix ($i=0,1,2$; $k=1,2,...$), that is,
$$
Q^*_3=\left(
              \begin{array}{ccccc}
                q_{01} & q_{02} & \ldots & q_{0k} & \ldots \\
                q_{11} & q_{12} & \ldots & q_{1k} & \ldots \\
                q_{21} & q_{22} & \ldots & q_{2k} & \ldots \\
              \end{array}
      \right),
$$
where $q_{ik}>0$, $q_{0k}+q_{1k}+q_{2k}=1$. The matrix defines the $Q^*_3$-representation of numbers
$$
[0;1]\ni x=\beta_{\alpha_1(x)1}+\sum^{\infty}_{k=2}\left[\beta_{\alpha_k(x)k}\prod^{k-1}_{j=1}q_{\alpha_j(x)j}\right]\equiv
\Delta^{Q^*_3}_{\alpha_1\alpha_2...\alpha_k...},
$$
where $\beta_{0k}=0$, $\beta_{1k}=q_{0k}=\beta_{0k}+q_{0k}$ and $\beta_{2k}=q_{0k}+q_{1k}=\beta_{1k}+q_{1k}$.

Note that when $q_i\equiv q_{i1}=q_{i2}=...=q_{ik}=...$ where $i\in A_3$ and $k\in N$ we obtain a $Q_3$-representation determined by a triple of positive numbers  $q_0$, $q_1$, $q_2$. In the case $q_0=q_1=q_2=\dfrac{1}{3}$ this representation coincides with the classical ternary representation.

Let a sequence of real numbers $(\varepsilon_k)$ also be given such that $0\leqslant\varepsilon_k\leqslant 1$; let $g_{0k}=\dfrac{1+\varepsilon_k}{3}=g_{2k}$ and $g_{1k}=\dfrac{1-2\varepsilon_k}{3}$, $k\in N$.

The object of our study is the function $f$ defined by the equality
$$
f(x)=\delta_{\alpha_1(x)1}+\sum^{\infty}_{k=2}\left[\delta_{\alpha_k(x)k}\prod^{k-1}_{j=1}g_{\alpha_j(x)j}\right] \equiv \Delta^{G^*_3}_{\alpha_1\alpha_2\ldots\alpha_k\ldots},\eqno(1)
$$
where $\delta_{0k}=0$, $\delta_{1k}=g_{0k}=\delta_{0k}+g_{0k}$ and $\delta_{2k}=g_{0k}+g_{1k}=\delta_{1k}+g_{1k}$.

Nowhere monotone functions with fractal properties have appeared in a number of papers. In   paper \cite{Pr1} authors examined non-monotone continuous on the interval $[0;1]$ functions whose derivative equals zero almost everywhere (in the sense of the Lebesgue measure). The authors defined these functions in terms of the $s$--adic representation and its generalization --- the $Q_s$--representation of real numbers and investigated properties of their graphs. Authors paid considerable attention to functions whose arguments admit ternary and quaternary representations. In paper \cite{PrSvyn} authors introduced a class of continuous functions with complicated local structure defined in terms of the $Q^*_s$--representation, whose properties depend on a prescribed finite set of real numbers. Papers \cite{Okamoto1, Okamoto2} studied a family of continuous strictly increasing singular functions depending on a parameter $0<a<1$ and defined via the ternary representation of the argument. The authors showed that, depending on the value of the parameter the function becomes nowhere differentiable, differentiable almost nowhere (in the sense of the Lebesgue measure), singular, etc.

\section{Correctness of the definition of the function}

To justify the correctness of the definition of the function $f$ we prove that series (1) converges for any sequence $(\alpha_n)\in L_3$ and that expression (1) yields the same value for different representations of the same $Q_3^*$--rational number.

Since for a series of modules the relations
$$|\delta_{\alpha_1(x)1}|+\sum^{\infty}_{k=2}|\delta_{\alpha_k(x)k}\prod^{k-1}_{j=1}g_{\alpha_j(x)j}|\leqslant
  \frac{2}{3}+\sum^{\infty}_{k=2}\left(\frac{2}{3}\right)^k=2$$
hold, then the series (1) is absolutely convergent.

Let $x$ be an arbitrary $Q_3^*$--rational point that has the following two representations:
$\Delta^{Q^*_3}_{\alpha_1\alpha_2\ldots\alpha_n(0)}$ and $\Delta^{Q^*_3}_{\alpha_1\alpha_2\ldots[\alpha_n-1](2)}$ with $\alpha_n\neq0$. We consider the difference

\begin{align*}
\rho\equiv & f(\Delta^{Q^*_3}_{\alpha_1\alpha_2\ldots[\alpha_n-1](2)})-
       f(\Delta^{Q^*_3}_{\alpha_1\alpha_2\ldots\alpha_n(0)})=  \\
   = & \sum^{\infty}_{k=n}\left[\delta_{\alpha_k(x_2)k}\prod^{k-1}_{j=1}g_{\alpha_j(x_2)j}\right] -
 \sum^{\infty}_{k=n}\left[\delta_{\alpha_k(x_1)k}\prod^{k-1}_{j=1}g_{\alpha_j(x_1)j}\right] \\
   =\!&\!\prod^{n-1}_{j=1}\negmedspace q_{\alpha_j\!j}\!
\left[\!\delta_{[\alpha_n-1]\!n}\!+\!
      \delta_{2[n+1]}\!g_{\alpha_n(x_2)\!n}\!+\!
      \delta_{2[n+2]}\!g_{\alpha_n(x_2)\!n}\!g_{\alpha_{n+1}\!(x_2)\![n+1]}\!+\!...\right.\\
   -&
\left.
      \delta_{0[n+1]} g_{\alpha_n(x_1)n}-
      \delta_{0[n+2]} g_{\alpha_n(x_1)n} g_{\alpha_{n+1}(x_1)[n+1]}-...\right]
       \end{align*}
 \begin{align*}
  {~} =\!&\!\prod^{n-1}_{j=1}\negmedspace q_{\alpha_j\!j}
       \!\left[\!\delta_{[\!\alpha_n-1]\!n}\!-\!\delta_{\alpha_n\!n}\!+\!g_{\alpha_n\!(x_2)\!n}
       \!\left(\!\delta_{2\![n+1]}\!+\!\delta_{2\![n+2]}\!g_{\alpha_{n+1}\!(x_2)\![n+1]}\!+\!...\right)\right]\\
   =\!&\!\prod^{n-1}_{j=1}\negmedspace q_{\alpha_j\!j}\!
       \left[\!\delta_{[\!\alpha_n-1\!]\!n}\!-\!\delta_{\alpha_n\!n}\!+\!g_{\alpha_n\!(x_2)\!n}\!
       \left(\!\delta_{2\![n+1]}\!+\!\sum^{\infty}_{k=2}\!\delta_{2\![n+k]}\!\prod^{k-1}_{j=1}\negmedspace g_{2[\!n+j\!]}\right)\!\right].
\end{align*}

Since $\delta_{2(n+1)}+\sum\limits^{\infty}_{k=2}\left[\delta_{2(n+k)}\prod\limits^{k-1}_{j=1}g_{2(n+j)}\right] =\Delta^{G^*_3}_{(2)}=1$, then

$$
  \rho=\prod^{n-1}_{j=1}g_{\alpha_jj}
\left[\delta_{[\alpha_n-1]n}-\delta_{\alpha_nn}+g_{[\alpha_n-1]n}
    \right].$$
If $\alpha_1=1$, then
$$\rho=
\prod\limits^{n-1}_{j=1}g_{\alpha_jj}\left[\delta_{0n}-\delta_{1n}+g_{0n}\right]=
\prod\limits^{n-1}_{j=1}g_{\alpha_jj}\left[0-g_{0n}+g_{0n}\right]=0.$$
If $\alpha_1=2$, then
$$\rho=
\prod\limits^{n-1}_{j=1}g_{\alpha_jj}\left[\delta_{1n}-\delta_{2n}+g_{1n}\right]=
\prod\limits^{n-1}_{j=1}g_{\alpha_jj}\left[g_{0n}-g_{0n}-g_{1n}+g_{1n}\right]=0.$$

Taking these facts into account, we conclude that the function $f$ is defined correctly.

\section{The set of values of the function}

From the definition of the function we obtain
$$f(0)=f\left(\Delta^{Q^*_3}_{(0)}\right)=
\delta_{01}+\sum\limits^{\infty}_{k=2}\left(\delta_{0k}\prod\limits^{k-1}_{j}g_{0j}\right)\!=\!
0+\sum\limits^{\infty}_{k=2}\left(0\prod\limits^{k-1}_{j}g_{0j}\right)\!=\!0;$$

$$f(\!1\!)\!=\!f\!\left(\!\Delta^{Q^*_3}_{(2)}\!\right)\!=\!
\delta_{21}\!+\!\sum\limits^{\infty}_{k=2}\!\left(\!\delta_{2k}\!\prod\limits^{k-1}_{j}g_{2j}\!\right)\!=\!
\frac{2\!-\!\varepsilon_1}{3}\!+\!\sum\limits^{\infty}_{k=2}\!\left(\!\frac{2\!-\!\varepsilon_k}{3}\!\prod^{k-1}_{j=1}\!\frac{1\!+\!\varepsilon_j}{3}\!\right)\!.$$

We estimate the expression $\dfrac{2-\varepsilon_1}{3}+\sum\limits^{\infty}_{k=2}\left(\dfrac{2-\varepsilon_k}{3}\prod\limits^{k-1}_{j=1}\frac{1+\varepsilon_j}{3}\right)=A$. Since $0\leqslant\varepsilon_n \leqslant 1$ then
$$
A\geqslant\frac{2}{3}+\sum^{\infty}_{k=2}\frac{2}{3}\prod^{k-1}_{j=1}\frac{1}{3}=
\frac{2}{3}+\sum^{\infty}_{k=2}\frac{2}{3}\cdot\frac{1}{3^{k-1}}=
\frac{2}{3}+\sum^{\infty}_{k=2}\frac{2}{3^k}=
\frac{2}{3}+\frac{2}{3^2}+...=1;
$$
$$
A\leqslant\frac{1}{3}+\sum^{\infty}_{k=2}\!\frac{1}{3}\prod^{k-1}_{j=1}\frac{2}{3}\!=\!
\frac{1}{3}+\!\sum^{\infty}_{k=2}\!\frac{1}{3}\left(\frac{2}{3}\right)^{k-1}\!=
\frac{1}{3}+\frac{1}{3}\left(\frac{2}{3}+\frac{2^2}{3^2}+...\right)=1.
$$
Therefore, $f(1)=1$.

\begin{lemma}
The function $f$, defined by equality (1), takes values in the interval $[0;1]$.
\end{lemma}

\begin{proof}
Let us consider the partial sum
$$
S_m\equiv\delta_{\alpha_11}+\sum^{m}_{k=2}\left[\delta_{\alpha_k(x)k}\prod^{k-1}_{j=1}g_{\alpha_j(x)j}\right], $$
and prove that $0\leqslant S_m<1$ with $S_m=0$ if and only if $\alpha_1=\alpha_2=\ldots=\alpha_m=0$.
To do this, we use the method of mathematical induction.

For $m=1$ we have $S_1=\delta_{\alpha_11}$. If $\alpha_1=0$ then $S_1=0$, if $\alpha_1\neq0$ then $0\leqslant S_1<1$ by the definition $\delta_{i1}$ for $i=1,2$.

For $m=2$ we have $S_2=\delta_{\alpha_11}+\delta_{\alpha_22}g_{\alpha_11}$. Since $\alpha_n\in\{0,1,2\}$, then depending on the values of the digits $\alpha_1,\alpha_2$ we obtain the following values for the partial sum $S_2$ in each case:
\[
\begin{array}{cc}
  \alpha_1=0: & \alpha_1=1: \\
  \begin{cases}
0 \text{ for } \alpha_2=0,\\
g_{02}g_{01} \text{ for } \alpha_2=1,\\
(g_{02}+g_{12})g_{01},  \alpha_2=2.
\end{cases} &
\begin{cases}
g_{01} \text{ for } \alpha_2=0,\\
g_{01}+g_{02}g_{11} \text{ for } \alpha_2=1,\\
g_{01}+(g_{02}+g_{12})g_{11},  \alpha_2=2.
\end{cases}
\end{array}\]

\[\alpha_1=2: \\
\begin{cases}
g_{01}+g_{11} \text{ for } \alpha_2=0,\\
g_{01}+g_{11}+g_{02}g_{21} \text{ for } \alpha_2=1,\\
g_{01}+g_{11}+(g_{02}+g_{12})g_{21} \text{ for } \alpha_2=2.
\end{cases}
\]

In each of these cases the number $S_2$ is nonnegative (and it equals zero only in the case $\alpha_1=\alpha_2=0$) and less than $1$. Hence, 
$$0\leqslant S_2=\delta_{\alpha_11}+\delta_{\alpha_22}g_{\alpha_11}<1.$$

Assume that the statement holds for $m=n$, that is, $0\leqslant S_n<1$, and prove that this implies the validity of the statement for $m=n+1$.

Consider the partial sum
\begin{align*} S_{n+1}&=\delta_{\alpha_11}+\sum^{n+1}_{k=2}\left[\delta_{\alpha_k(x)k}\prod^{k-1}_{j=1}g_{\alpha_j(x)j}\right]\\
&=\delta_{\alpha_11}+g_{\alpha_11}\left(\delta_{\alpha_22}+\sum^{n+1}_{k=3}\left[\delta_{\alpha_k(x)k}\prod^{k-1}_{j=1}g_{\alpha_j(x)j}\right]\right)\\ &=\delta_{\alpha_11}+g_{\alpha_11}S'_n,
\end{align*}
where  $S'_n=\delta_{\alpha_22}+\sum\limits^{n+1}_{k=3} \left[\delta_{\alpha_k(x)k}\prod\limits^{k-1}_{j=1}g_{\alpha_j(x)j}\right]$ denotes the partial sum corresponding to the set $(\alpha_2,\alpha_3,\ldots,\alpha_{n+1})$.

By the assumption, $0\leqslant S'_n<1$ and $S'_n=0$ if and only if $\alpha_2=\alpha_3=\ldots=\alpha_{n+1}=0$. If $S'_n=0$, then $1> S_{n+1}=\delta_{\alpha_11}\geqslant 0$ (moreover, $S_{n+1}=0$ when $\alpha_1=0$). If $0< S'_n<1$, then
$$0<S_{n+1}<
\begin{cases}
g_{01},\text{\,\,\,if\,\,\,} \alpha_1=0,\\
g_{01}+g_{11},\text{\,\,\,if\,\,\,}\alpha_1=1,\\
g_{01}+g_{11}+g_{21},\text{\,\,\,if\,\,\,}\alpha_1=2.
\end{cases}
$$

Thus, for any $x\in[0;1]$ and $m\in N$ the inequality $0\leqslant S_m<1$ holds.

Passing to the limit in the last double inequality, we obtain 
$$0\leqslant f(x)=\lim\limits_{m\to\infty}S_m\leqslant1=f(1).$$
\end{proof}

\section{Continuity of the function}

\begin{theorem}
The function $f$, defined by equality (1), is continuous at every point of the interval $[0;1]$.
\end{theorem}

\begin{proof}
To prove the continuity of the function at the point $x_0=\Delta^{Q^*_3}_{\alpha_1\alpha_2\ldots\alpha_n\ldots}$ we use the definition of continuity; namely, we prove that $\lim\limits_{x\to x_0}|f(x)-f(x_0)|=0$.
If $x_0$ is a $Q^*_3$--irrational point, then $x\to x_0$ is equivalent to $m\to\infty$, where $\alpha_m(x)\neq\alpha_m(x_0)$, but $\alpha_j(x)=\alpha_j(x_0)$ for $j < m$. Then
\begin{align*}
f(x)-f(x_0)=&\prod\limits^{m-1}_{j=1}g_{\alpha_j(x_0)j}\times\\
&\times\left(\delta_{\alpha_m(x)m}+\sum^{\infty}_{k=m+1}\left[\delta_{\alpha_k(x)k}\prod^{k-1}_{j=m+1}g_{\alpha_j(x)j}\right]\right.-\\
&-  \left.           \delta_{\alpha_m(x_0)m}+\sum^{\infty}_{k=m+1}\left[\delta_{\alpha_k(x_0)k}\prod^{k-1}_{j=m+1}g_{\alpha_j(x_0)j}\right]\right)=\\
&=\prod\limits^{m-1}_{j=1}g_{\alpha_j(x_0)j}\cdot B \to 0\,\,\, (m\to \infty).
\end{align*}
Since $\prod\limits^{m-1}_{j=1}g_{\alpha_j(x_0)j}\to 0$ as $m\to\infty$, and the series
$$\delta_{\alpha_m(u)m}+\sum\limits^{\infty}_{k=m+1}\left[\delta_{\alpha_k(u)k}\prod\limits^{k-1}_{j=m+1}g_{\alpha_j(u)j}\right]$$
converges for $u = x$ and $u = x_0$ .
To prove the continuity of $f$ at a $Q^*_3$--rational point $x_0$, it suffices to use the same approach. However, to prove left-side continuity we use the representation of the number $x_0$ with period (2), and to prove right-side continuity we use the representation with period (0).
\end{proof}

\begin{corollary}
The set of values of the function $f$, defined by equality (1), is the interval $[0;1]$.
\end{corollary}

\section{Monotonicity Properties}

\begin{theorem}
The increment $\mu_f(\Delta^{Q^*_3}_{c_1\ldots c_m})$ of the function $f$, defined on the interval $[0;1]$ by equality (1), is calculated by the formula $$\mu_f(\Delta^{Q^*_3}_{c_1\ldots c_m})=\prod^{m}_{j=1}g_{c_j j}.$$
\end{theorem}

\begin{proof} In fact,
\begin{align*}
\mu_f(\Delta^{Q^*_3}_{c_1c_2\ldots c_m})&=
f(\Delta^{Q^*_3}_{c_1c_2\ldots c_m(2)})-f(\Delta^{Q^*_3}_{c_1c_2\ldots c_m(0)})=\\
& \sum^{\infty}_{k=m+1}\delta_{2k}\prod^{k-1}_{j=1}g_{\alpha_j(x_2)j}-\sum^{\infty}_{k=m+1}\delta_{0k}\prod^{k-1}_{j=1}g_{\alpha_j(x_1)j}=\\
& \prod^{m}_{j=1}g_{c_jj}\left(\sum^{\infty}_{k=m+1}\delta_{2k}\prod^{k-1}_{j=1}g_{2j}-\sum^{\infty}_{k=m+1}\delta_{0k}\prod^{k-1}_{j=1}g_{0j}\right)=\\
&
\prod^{m}_{j=1}g_{c_jj}\left(\delta_{2(m+1)}+\sum^{\infty}_{k=m+2}\delta_{2k}\prod^{k-1}_{j=m+1}g_{2j}\right)=
\prod^{m}_{j=1}g_{c_jj}.
\end{align*}
\end{proof}

\begin{corollary}
If $\varepsilon_k=\dfrac{1}{2}$ for some natural $k\leqslant m$, then $g_{1k}=0$ and $g_{0k}=g_{2k}=\dfrac{1}{2}$. Hence, the function $f(x)$ is constant on some cylinder $\Delta^{Q^*_3}_{c_1c_2\ldots c_m}$.
\end{corollary}

\begin{corollary}
If $0\leqslant \varepsilon_k<\dfrac{1}{2}$, then the function $f(x)$ is strictly increasing on entire domain of its definition.
\end{corollary}

\begin{theorem}
The function $f$, continuous on $[0;1]$ and defined by equality (1), is nowhere monotonic if the inequality $\dfrac{1}{2}<\varepsilon_k \leqslant 1$ holds for every $k\in N$.
\end{theorem}
\begin{proof}
Since $\varepsilon_k\neq\dfrac{1}{2}$ then we have $g_{0k}g_{1k}g_{2k}\neq 0$. Therefore the function $f$ does not contain any constant intervals. Suppose that the function $f$ has an interval of monotonicity $(a;b)\subset[0;1]$.
Since a cylinder $\Delta^{Q^*_3}_{c_1\ldots c_m}\subset (a;b)$ always exists, it also serves as an interval of monotonicity of the function $f$.
Since $g_{0k}g_{1k}g_{2k}\neq 0$, then $\mu_f(\Delta^{Q^*_3}_{c_1c_2\ldots c_m})=\prod\limits^{m}_{j=1}g_{c_j j}\neq 0$ and
$$\mu_f(\Delta^{Q^*_3}_{c_1c_2\ldots c_m0})\mu_f(\Delta^{Q^*_3}_{c_1c_2\ldots c_m1})\mu_f(\Delta^{Q^*_3}_{c_1c_2\ldots c_m2})<0,$$
that is, on the cylinder $\Delta^{Q^*_3}_{c_1c_2\ldots c_m1}$ the function $f$ has a negative increment, whereas on the cylinders $\Delta^{Q^*_3}_{c_1c_2\ldots c_m0}$ and $\Delta^{Q^*_3}_{c_1c_2\ldots c_m2}$ it has a positive increment. This contradicts the monotonicity of $f$ on the cylinder $\Delta^{Q^*_3}_{c_1c_2\ldots c_m}$.
\end{proof}

\section{Indications of nowhere monotonicity and Cantor-type structure}

If $\varepsilon_n=\dfrac{1}{2}$, then the function $f(x)$ is constant on the cylindrical interval $\nabla^{Q_3^*}_{c_1...c_{m-1}1}$, where $c_i \in V$ for $i = 0,1,...,m-1$ and $m\in N$, which is adjacent to the set $C[Q_3^*,V]$.

Since, for any $x$ from this interval, we have $\alpha_m(x)=2$, then
$$\prod^{m}_{j=1}g_{\alpha_jj}(x)=0,$$
and
\begin{align*}f(x)=&\delta_{\alpha_1(x)1}+\delta_{\alpha_2(x)2}g_{\alpha_1(x)1}+...+
        \delta_{\alpha_{m-1}(x)[m-1]}\prod^{m-2}_{j=1}g_{\alpha_j(x)j}+\\
        &0\cdot\prod^{m-1}_{j=1}g_{\alpha_j(x)j}+0=const,
\end{align*}
that is, $f$ is constant on this interval.

The sum of the lengths of all intervals $\nabla^{Q_3^*}_{c_1...c_{m-1}1}$, adjacent to the set $C[Q_3^*,V]$, where $c_i \in V$ for $i = 0,1,...,m-1$ and $m\in N$, equals 1. Indeed,
$$
S=|\Delta^{Q_3^*}_{1}|+\!\sum_{c_1 \in V}|\Delta^{Q_3^*}_{c_11}|+\!\sum_{c_1 \in V}\!\sum_{c_2 \in V}|\Delta^{Q_3^*}_{c_1c_21}|+...=\frac{1}{3}+\frac{2}{3^2}+\frac{2^2}{3^3}+...=1.
$$
Since the function $f$ is constant on each of these intervals $\nabla^{Q_3^*}_{c_1...c_{m-1}1}$, $c_i \in V$ for $i= 0,1,...,m-1$ and $m\in N$,
then its derivative equals 0 on each of them. Consequently, the function is singular.

\section{Function extrema}

We denote the increment $\mu_f$ of the function $f$ on the cylinder $\Delta^{Q_3^*}_{c_1...c_m}$ by $\mu_m$, that is
$$\mu_m=\mu_f(\Delta^{Q_3^*}_{c_1...c_m})=\prod^m_{j=1}g_{c_j j}(x).$$

\begin{lemma}
The function $f$ on the cylinder $\Delta^{Q_3^*}_{c_1...c_m}$ attains its maximum and minimum values at its endpoints. Moreover, 
$$\text{ if }\mu_m>0,\text{   then   }\max f(x)=f(\Delta^{Q_3^*}_{c_1...c_m(2)}),
                         \min f(x)=f(\Delta^{Q_3^*}_{c_1...c_m(0)}),$$
$$\text{ if }\mu_m<0,\text{   then   }\max f(x)=f(\Delta^{Q_3^*}_{c_1...c_m(0)}),
                         \min f(x)=f(\Delta^{Q_3^*}_{c_1...c_m(2)}).$$
\end{lemma}
\begin{proof}
Since the cylinder $\Delta^{Q_3^*}_{c_1...c_m}$ is a segment $[\Delta^{Q_3^*}_{c_1...c_m(0)}$;$\Delta^{Q_3^*}_{c_1...c_m(2)}]$, then we can express the value of the function $f$ at an arbitrary point $x\in\Delta^{Q_3^*}_{c_1...c_m}$ as
$$f(x)=f(\Delta^{Q_3^*}_{c_1...c_m(0)})+\mu_m\cdot f(\Delta^{Q_3^*}_{\alpha_{m+1}\alpha_{m+2}...\alpha_{m+n}...}).$$
If $\alpha_{m+j}=0$ for any $j\in N$, then $f(\Delta^{Q_3^*}_{\alpha_{m+1}\alpha_{m+2}...\alpha_{m+n}...})=0$. If $\alpha_{m+j}=2$ for any $j\in N$, then $f(\Delta^{Q_3^*}_{\alpha_{m+1}\alpha_{m+2}...\alpha_{m+n}...})=1$.\\
Then if $\mu_m>0$, we have $$\max_{x\in\Delta^{Q_3^*}_{c_1...c_m}} f(x)=f(\Delta^{Q_3^*}_{c_1...c_m(2)}),
                         \min_{x\in\Delta^{Q_3^*}_{c_1...c_m}} f(x)=f(\Delta^{Q_3^*}_{c_1...c_m(0)}),$$
and if $\mu_m<0$, we have $$\max_{x\in\Delta^{Q_3^*}_{c_1...c_m}} f(x)=f(\Delta^{Q_3^*}_{c_1...c_m(0)}),
                         \min_{x\in\Delta^{Q_3^*}_{c_1...c_m}} f(x)=f(\Delta^{Q_3^*}_{c_1...c_m(2)}).$$
\end{proof}

\section{Properties of the graph $\Gamma_f$ of the function}
\begin{theorem}
If $\varepsilon_n=const$ and $\prod\limits_{i=0}^2g_i\neq 0$, then the graph of the function $f$ and its subset $\Gamma^i_f\equiv\{(x;y): x\in\Delta^{G^*_3}_i, y=f(x)\}$ are affinely equivalent with $\Gamma^i_f=\phi_i(\Gamma_f)$, where
\begin{equation*}
\phi_i :
\begin{cases}
x'=q_ix+\beta_i,\\
y'=g_iy+\delta_i, \,\,\,\,\,\,i\in\{0,1,2\}
\end{cases}
\end{equation*}
\end{theorem}

\begin{proof}
To prove the theorem we show that $\Gamma_{\phi}^i\subset \phi_i(\Gamma_f)$ and $\phi_i(\Gamma_f)\subset \Gamma_{\phi}^i$.

Let $M(x;y)\subset\Gamma_f^i$, that is $$x=\Delta^{Q_3}_{i\alpha_1\alpha_2\ldots\alpha_n\ldots},\text{ and } y=\delta_i+g_i \left(\delta_{\alpha_1}+ \sum\limits^{\infty}_{k=2}\left(\delta_{\alpha_k}\prod\limits^{k-1}_{j=1}g_{\alpha_j}\right)\right).$$
Let us select the point $M_0(x_0;y_0)$ on the graph $\Gamma_f$, where $$x_0=\Delta^{Q_3}_{\alpha_1\alpha_2\ldots\alpha_n\ldots},\text{ and } y_0=\delta_{\alpha_1}+ \sum\limits^{\infty}_{k=2}\left(\delta_{\alpha_k}\prod\limits^{k-1}_{j=1}g_{\alpha_j}\right).$$
Its image $\phi_i(M_0)=M$, hence $M\in \Gamma_f$.

Let $M'(x';y')$ be the image of the point $M(x,y)$ under the affine transformation $\phi_i$, that is
$$x=\Delta^{Q_3}_{\alpha_1\alpha_2\ldots\alpha_n\ldots},\text{ and } y=\delta_{\alpha_1}+ \sum\limits^{\infty}_{k=2}\left(\delta_{\alpha_k}\prod\limits^{k-1}_{j=1}g_{\alpha_j}\right),$$ and
$$x'=\Delta^{Q_3}_{i\alpha_1\alpha_2\ldots\alpha_n\ldots},\text{ and } y'=\delta_i+g_i \left(\delta_{\alpha_1}+ \sum\limits^{\infty}_{k=2}\left(\delta_{\alpha_k}\prod\limits^{k-1}_{j=1}g_{\alpha_j}\right)\right).$$
It follows that $M'\in\Gamma_f^i$, and therefore $\phi_i(\Gamma_f)\subset\Gamma^i_f$. Thus, $\Gamma^i_f=\phi_i(\Gamma_f)$.
\end{proof}

\section{Level sets of the function}

Recall that the level set $y_0$ of a function $f$ is the set
$$f^{-1}(y_0)=\{x:f(x)=y_0\}.$$

If $0\leqslant \varepsilon_n<\dfrac{1}{2}$ for all $n\in N$, then the function $f$ is strictly increasing, so each of its level sets consists of a single point.

If $\varepsilon_n=\dfrac{1}{2}$, then $g_{1n}=0$ for all $n\in N$ and each level set of the function $f$, due to its continuity, is either a point or a segment.

If $\dfrac{1}{2}<\varepsilon_n\leqslant 1$, then  $g_{1n}<0$ and each level set of the function $f$ is a countable set. Since the increments of the function $f$ on the cylinders $\Delta^{Q_3^*}_{\alpha_1(x)\alpha_2(x)...\alpha_i(x)}$ of rank $i$ and $\Delta^{Q_3^*}_{\alpha_1(x)\alpha_2(x)...\alpha_i(x)\alpha_{i+1}(x)}$
of rank $i+1$ have opposite signs, and taking into account the continuity of $f$, the line $y = y_0$ intersects the graph of $f$ at least at two points that belong to the cylinder $\Delta^{Q_3^*}_{\alpha_1(x)\alpha_2(x)...\alpha_i(x)}$
 and do not belong to the cylinder $\Delta^{Q_3^*}_{\alpha_1(x)\alpha_2(x)...\alpha_i(x)\alpha_{i+1}(x)}$.
Therefore, the level set $f^{-1}(y_0)$ is a countable set. It cannot be continuous, since the set of local maxima and minima is countable.

\end{document}